\documentclass[12pt]{article}
\usepackage{setspace}
\usepackage{amsthm,amsfonts,amssymb,epsfig,graphics,amsmath,amsbsy,cite,subfigure}
\usepackage{color}		
\usepackage{epsfig}

\newcommand{\intL}{\int\limits }
\newcommand{\intR}{\int\limits_{\mathbb{R}} }

\newcommand{\RR}{\mathbb{R}}
\newcommand{\half}{^\infty_0}

\usepackage{amssymb}
\usepackage{amsmath}
\usepackage{graphicx}
\newtheorem{thm}{Theorem}

\newtheorem{rmk}[thm]{Remark}

\newtheorem{prop}[thm]{Proposition}
\title{On the Determination of a Function from an elliptical Radon Transform}
\author{Sunghwan Moon\thanks{Mathematics Department, Texas A\&M University, College Station, Tx 77843-3365\newline
\indent  E-mail address: shmoon@math.tamu.edu}}
\date{}
\begin{document}
\maketitle

\begin{abstract}
 In recent years, Radon type transforms that integrate functions over various sets of ellipses/ellipsoids have been considered
in SAR, ultrasound reflection tomography, and radio tomography. In this paper, we consider the transform that integrates a given function 
in $\RR^n$ over a set of solid ellipsoids of rotation with a fixed eccentricity and foci restricted to a hyperplane. Inversion 
formulas are obtained for appropriate classes of functions that are
even with respect to the hyperplane. Stability estimates, range conditions and local uniqueness results
are also provided. 
\end{abstract}

\section{Introduction}

Radon-type transforms that integrate functions over various sets of ellipses/ellipsoids have been arising in 
the recent decade, due to
studies in synthetic aperture radar (SAR) \cite{ambartsoumianfknq11,cokert07,krishnanq11,krishnanlq12}, ultrasound reflection tomography \cite{gouiaa12,ambartsoumiankq11}, and 
radio tomography \cite{wilsonp10,wilsonp09,wilsonpv09}. In particular, radio tomography is a new imaging 
method, which uses a wireless network of radio transmitters and receivers to image the distribution of 
attenuation within the network. The usage of radio frequencies brings in significant non-line-of-sight 
propagation, since waves propagate along many paths from a transmitter to a receiver.
Given a transmitter and a receiver, wave paths observed for a given duration are all contained in an 
ellipsoid with foci at these two devices. It was thus suggested in~\cite{wilsonp10,wilsonp09,wilsonpv09} to 
approximate the obtained signal by the volume integral of the attenuation over this ellipsoid, which is 
the model we study in this article.

Due to these applications, there have been several papers devoted to such ``elliptical Radon transform.'' 
The family of ellipses with one focus fixed at the origin and the other one moving along a given line was 
considered in~\cite{krishnanlq12}. In the same paper, the family of ellipses with a fixed focal distance was 
also studied. 
The authors of~\cite{gouiaa12,ambartsoumiankq11} dealt with the case of circular acquisition, when the two foci of ellipses 
with a given focal distance are located on a given circle. A family of ellipses with two moving foci was 
also handled in~\cite{cokert07}. 

In all these works, however, the ellipses have varying eccentricity. Also, their data were the line 
integrals of the function over ellipses rather than area integrals. The radio tomography application makes it
interesting to consider integrals over solid ellipsoids.  
In this article, we consider the volume integrals of an unknown attenuation function over the family of 
ellipsoids of rotation in $\RR^n$ with a fixed eccentricity and two foci located in a given hyperplane.
We thus reserve the name \textbf{elliptical Radon transform} $R_Ef$ for the volume integral of a 
function $f$ over this family of ellipsoids.

The volume integral of a function $f(x)$ over an ellipsoid of the described type is equal to zero if the 
function is odd with respect to the chosen hyperplane. If the hyperplane is given by $x_n=0$,
we thus assume the function $f(x):\RR^n\rightarrow \RR$ to be even with respect to $x_n$: $f(x',x_n)=f(x',-x_n)$ 
where $x=(x',x_n)\in \RR^{n-1}\times \RR$. 

Given a Radon type transform, one is usually interested, among others, in the following questions:
uniqueness of reconstruction, inversion formulas and algorithms, range conditions, and a stability estimate~\cite{natterer01,nattererw01}. These are the issues we address below.

The problem is stated precisely in section~\ref{formulation}. 
Two inversion formulas are presented in sections~\ref{natinversion} and~\ref{reddinginversion}. 
Analogue of the Fourier slice theorem is obtained in section~\ref{natinversion} by taking the Fourier transform with respect to the center and a radial Fourier transform with respect to the half distance between two foci. 
This theorem plays a critical role in getting a stability estimate and necessary range conditions. 
The formula discussed in section~\ref{reddinginversion} is obtained by taking a Fourier type transform and needs less integration than the previous one in section~\ref{natinversion}. A stability estimate is handled in sections~\ref{estimate}. Section~\ref{unique} is devoted to uniqueness for a local 
data problem. 

\section{Formulation of the problem}\label{formulation}

We consider all solid ellipsoids of rotation in $\RR^n$ with a fixed eccentricity $1/\lambda$, where $\lambda>1$ 
and foci located in the hyperplane $x_n=0$. We will identify this hyperplane with $\RR^{n-1}$. The set of such ellipsoids depends upon $2n-2$ parameters, which is 
$n-2$ too many. To reduce the overdeterminacy, we require that the focal axis is parallel to a given line, for instance,
the $x_1$ coordinate axis. 

Let $u\in \RR^{n-1}$ be the center of such an ellipsoid and let $t>0$ be the half of the focal distance. We denote this ellipsoid by 
$E_{u,t}$. Then, the foci are 
$$
c_1=(u_1+t,u_2,\cdots,u_{n-1},0) \mbox{ and } c_2=(u_1-t,u_2,\cdots,u_{n-1},0)
$$
and the points $x\in E_{u,t}$ are described as follows:
$$
\dfrac{(x_1-u_1)^2}{\lambda^2}+\dfrac{(x_2-u_2)^2}{\lambda^2-1}+\cdots+\dfrac{x_n^2}{\lambda^2-1}\leq t^2.
$$
To shorten the formulas, we are going to use the following notation:
\begin{equation*}
 \nu:=\sqrt{\lambda^2-1}.
\end{equation*}
The elliptical Radon transform $R_E$ maps a locally integrable function $f(x)$ into its integrals over the solid ellipsoids
$E_{u,t}$ for all $u\in\RR^{n-1}$ and $t>0$:
\begin{equation*}
 \begin{array}{l}
R_Ef(u,t)=\displaystyle \int\limits_{E_{u,t}} f(x)dx.
 \end{array}
\end{equation*}
Our goals are to reconstruct $f$ from $R_Ef$ and to study properties of this transform.

\section{Inversion of the elliptical Radon transform}\label{natinversion} 
In this section, we assume $f\in C^\infty_c(\RR^n)$. 
Here is our strategy.
First of all, we change the ellipsoid volume integral to the ellipsoid surface integral, differentiating with respect to $t$. Second, we take the Fourier transform of this derivative of $R_Ef$ with respect to $u$. Next, taking a radial Fourier transform with respect to $t$, we obtain an analogue of the Fourier slice theorem.

We introduce a back projection operator $R^*_E$ for $g(u,t)\in C^\infty_c(\RR^{n-1}\times \RR_+)$ as 
\begin{equation}\label{backproj}
R_E^* g(x)=\displaystyle\int\limits_{\RR^{n-1}} g\left(u, \sqrt{\frac{|u_1-x_1|^2}{\lambda^2}+\frac{|\tilde{u}-\tilde{x}|^2}{\nu^2}+\frac{x_n^2}{\nu^2}}\right)du.
\end{equation}
In fact, $R_E^* g(x)$ is the dual transform not to $R_Ef(u,t)$, but rather to $\frac{\partial}{\partial t}R_Ef(u,t)$, i.e., 
\begin{equation}\label{eq:dual}
\displaystyle \int\limits^\infty_0\int\limits_{\RR^{n-1}}\dfrac{\partial}{\partial t}R_Ef(u,t)g(u,t)dudt=C(\lambda)\int\limits_{\RR^n}f(x)R_E^*g(x)dx.
\end{equation}

Let $\chi_{S}$ denote the characteristic function of a set $S\subset \RR^n$:
$$
\chi_S(x)=\left\{\begin{array}{ll}1, &\mbox{if } x\in S, \\ 0,&\mbox{otherwise.}\end{array}\right.
$$ 
Then the elliptical Radon transform can be written as 
 \begin{equation}\label{ellipticalradon}
\begin{array}{ll}
R_Ef(u,t)&=\displaystyle \int\limits_{\RR^n} \chi_{E_{u,t}} f(x)dx=C(\lambda)\displaystyle \int\limits_{\RR^n}
\chi_{|x|<t} f(\lambda x_1+u_1,\nu \tilde{x}+\tilde{u},\nu x_n) dx\\
&=C(\lambda)\displaystyle \int\limits^t_0r^{n-1}\int\limits_{S^{n-1}} f(\lambda r y_1+u_1,\nu r
\tilde{y}+\tilde{u},r\nu y_n) d\sigma (y)dr,\end{array}
\end{equation}
where $u=(u_1,\tilde{u})\in \RR^{n-1}$, $x=(x_1,\tilde{x},x_n)\in\RR^n$, 
$C(\lambda)=\lambda\nu^{n-1}$ is the Jacobian factor, and $\sigma(y)$ is the surface measure on $S^{n-1}$. 

Formula~\eqref{ellipticalradon} can be simplified by differentiation with respect to $t$ and division by $t^{n-1}$, which yield
\begin{equation}\label{differREF}
\begin{array}{ll}
\dfrac{1}{t^{n-1}}\dfrac{\partial}{\partial t}R_Ef(u,t)&=C(\lambda)\displaystyle \int\limits_{|y|=1} 
f(\lambda t y_1+u_1,\nu t\tilde{y}+\tilde{u},t\nu y_n) d\sigma (y)\\
&=2C(\lambda)\displaystyle \int\limits_{|y'|\leq 1} f(u+(t\lambda y_1,t\nu 
\tilde{y}),t\nu \sqrt{1-|y'|^2})\frac{dy'}{\sqrt{1-|y'|^2}},
\end{array}
\end{equation}
where $y'=(y_1,\tilde{y})\in \RR^{n-1}$.

It is easy to check that $R_E $ is invariant under the shift with respect to the first $n-1$ variables. That is, 
if $f_a(x):=f(x'+a,x_n)$ for $x=(x',x_n)\in \RR^n$ and $a\in \RR^{n-1}$, we have 
$$
(R_E f_a)(u,t)=(R_Ef)(u+a,t).
$$
Thus, application of the $(n-1)$-dimensional Fourier Transform with respect to the center $u$ seems reasonable.
 Doing this and changing the 
variable $y'\in\RR^{n-1}$ to the polar coordinates $(\theta, s)\in S^{n-1}\times[0,\infty)$, we get 
$$
\dfrac{1}{t^{n-1}}\dfrac{\partial}{\partial t}\widehat{R_Ef}(\xi',t)=
\displaystyle 2C(\lambda)\int\limits^1_0\frac{s^{n-2}}{\sqrt{1-s^2}}\hat{f}(\xi',t\nu \sqrt{1-s^2})
\int\limits_{S^{n-2}}e^{its(\lambda \theta_1,\nu \tilde{\theta})\cdot \xi'}d\theta ds,
$$
where $\widehat{R_Ef}$ and $\hat{f}$ are the Fourier Transforms of $R_Ef$ and $f$ with respect to the first 
$n-1$ coordinates $x'$ of $x$ and $u$ of $(u,t)$, respectively, and $\theta=(\theta_1,\tilde{\theta})\in S^{n-2}$.

To compute the inner integral, we use the known identity \cite{andersson88,fawcett85}
$$
\displaystyle \int\limits_{S^{n-1}}e^{i\xi\cdot\theta}d\theta=(2\pi)^{n/2}|\xi|^{(2-n)/2}J_{(n-2)/2}(|\xi|).
$$
We thus get 
$$
\begin{array}{l}
\dfrac{1}{t^{n-1}}\dfrac{\partial}{\partial t}\widehat{R_Ef}\left(\dfrac{\xi_1}{\lambda},
\dfrac{\tilde{\xi}}{\nu },t\right)= \omega_n\displaystyle \int\limits^1_0\frac{s^{n-2}}{\sqrt{1-s^2}}\hat{f}\left(\dfrac{\xi_1}{\lambda},
\dfrac{\tilde{\xi}}{\nu },t\nu \sqrt{1-s^2}\right) (ts|\xi'|)^{(3-n)/2}J_{(n-3)/2}
(ts|\xi'|)ds,
\end{array}
$$
where $\omega_n=2(2\pi)^{(n-1)/2}C(\lambda)$.

This enables us to get an analogue of the Fourier slice Theorem.
\begin{thm} \label{inversionfourier}
For a function $f\in C^\infty_c(\RR^n)$ that is even with respect to $x_n$, the following formula holds:
\begin{equation}\label{inversion}
\hat{f}(\xi)=\dfrac{|(\lambda\xi_1,\nu \tilde{\xi},\nu \xi_n)|^{n-2}|\nu\xi_n|}{2^{n+1}\pi^nC(\lambda)^2}\mathcal{F}\left(R_E^*\dfrac{1}{t^{n-1}}\dfrac{\partial}{\partial t}R_E f\right)(\xi), 
\end{equation}
where $\mathcal{F}f$ is the $n$-dimensional Fourier transform of $f$.
\end{thm}

\begin{proof}
Let us denote the radial Fourier transform by $H_nf(\rho)$, i.e., 
$$
H_nf(\rho):=\rho^{1-n/2}\int\limits^\infty_0 t^{n/2}J_{(n-2)/2}(t\rho)f(t)dt.
$$
We recall that if $f$ is a radial function on $\RR^n$, then the Fourier transform $\hat{f}$ of $f$ is also radial and $\hat{f}=H_nf_0$ where $f_0(|x|)=f(x)$(cf. ~\cite{strichartz03}).
Taking this transform of $\frac{1}{t^{n-1}}\frac{\partial}{\partial t}\widehat{R_Ef}$ as a function of $t$, we have for $\xi=(\xi_1,\tilde{\xi})=\xi'\in\RR^{n-1}$,
\begin{equation}\label{eq:radial}
\begin{array}{l}
H_n\left(\dfrac{1}{t^{n-1}}\dfrac{\partial}{\partial t}\widehat{R_E f}\right)\left(\dfrac{\xi_1}{\lambda},\dfrac{\tilde{\xi}}{\nu },\rho\right)\\
=\displaystyle \omega_n \rho^{1-n/2} \displaystyle\int\limits^\infty_0 \int\limits^1_0 
 t^{\frac{n}{2}}J_{\frac{n-2}{2}}(t\rho)(s|\xi'|)^{\frac{3-n}{2}}\displaystyle J_{\frac{n-3}{2}}(ts|\xi'|)\hat{f}\left(\dfrac{\xi_1}{\lambda},\dfrac{\tilde{\xi}}{\nu },t\nu \sqrt{1-s^2}\right) \dfrac{s^{n-2}dsdt}{\sqrt{1-s^2}}.
\end{array}
\end{equation}

It is known~\cite[p. 59 (18) vol.2 or for $n=2$, p.55 (35) vol.1]{batemann} that 
for $a>0, \beta>0,$ and, $\mu>\nu>-1$, 
$$
\begin{array}{l}
\displaystyle\int\limits\half x^{\nu+1/2}(x^2+\beta^2)^{-1/2\mu}J_\mu(a(x^2+\beta^2)^{1/2})J_\nu(xy)(xy)^{1/2}dx\\=\left\{\begin{array}{ll}a^{-\mu}y^{\nu+1/2}\beta^{-\mu+\nu+1}(a^2-y^2)^{1/2\mu-1/2\nu-1/2}J_{\mu-\nu-1}(\beta(a^2-y^2)^{1/2}) &\mbox{ if }0<y<a,\\
0&\mbox{ if } a<y<\infty.\end{array}\right.
\end{array}
$$ 
To use the above identity, we make the change of variables $(s,t)\rightarrow(x,\beta)$, where $t=\sqrt{x^2+\beta^2}$ and $s=x/\sqrt{x^2+\beta^2}$ in equation~\eqref{eq:radial}, which gives
\begin{equation}\label{eq:fourier}
\begin{array}{l}
H_n\left(\dfrac{1}{t^{n-1}}\dfrac{\partial}{\partial t}\widehat{R_E f}\right)\left(\dfrac{\xi_1}{\lambda},\dfrac{\tilde{\xi}}{\nu },\rho\right)\\= \omega_n \rho^{\frac{2-n}{2}}|\xi'|^{\frac{3-n}2}
\displaystyle\int\limits^\infty_0 \int\limits^\infty_0 |x|J_{\frac{n-2}{2}}(\rho(x^2+\beta^2)^{\frac{1}{2}})(x^2+\beta^2)^{-\frac{n-2}{4}}J_{\frac{n-3}{2}}(x|\xi'|)\hat{f}\left(\dfrac{\xi_1}{\lambda},\dfrac{\tilde{\xi}}{\nu },\nu \beta\right) dxd\beta\\
=\left\{\begin{array}{ll}\displaystyle C(\lambda)\frac{2^{n/2+1}\pi^{n/2}\rho^{2-n}}{\sqrt{\rho^2-|\xi'|^2}}\int\limits^\infty_0  \hat{f}\left(\dfrac{\xi_1}{\lambda},\dfrac{\tilde{\xi}}{\nu },\beta\nu \right)\cos(\beta\sqrt{\rho^2-|\xi'|^2}) d\beta&\mbox{if } |\xi'|<\rho,\\
0&\mbox{ otherwise.} \end{array}\right.
  \end{array} 
\end{equation}
Substituting $\rho=|\xi|$ yields 
$$
H_n\left(\dfrac{1}{t^{n-1}}\dfrac{\partial}{\partial t}\widehat{R_E f}\right)\left(\dfrac{\xi_1}{\lambda},\dfrac{\tilde{\xi}}{\nu },|\xi|\right)=\displaystyle C(\lambda)\frac{2^{n/2+1}\pi^{n/2}|\xi|^{2-n}}{|\xi_n|}\int\limits^\infty_0  \hat{f}\left(\dfrac{\xi_1}{\lambda},\dfrac{\tilde{\xi}}{\nu },\beta\nu \right)\cos(\xi_n\beta) d\beta.
$$
Since $f$ is even in $x_n$, the last integral is the Fourier transform of $f$ with respect to $x-n$, so we get
\begin{equation}\label{fourier}
\begin{array}{l}
H_n\left(\dfrac{1}{t^{n-1}}\dfrac{\partial}{\partial t}\widehat{R_E f}\right)\left(\dfrac{\xi_1}{\lambda},\dfrac{\tilde{\xi}}{\nu },|\xi|\right)
=\displaystyle \frac{2^{n/2+1}\pi^{n/2}|\xi|^{2-n}}{|\xi_n|}\lambda \nu^{n-2} \hat{f}\left(\dfrac{\xi_1}{\lambda},\dfrac{\tilde{\xi}}{\nu },\dfrac{\xi_n}{  \nu }\right).
\end{array} 
\end{equation}

Taking the Fourier transform of $R_E^*g$ with respect to $x$ yields
\begin{equation}\label{backfourier}
 \begin{array}{l}
\displaystyle\widehat{R_E^* g}\left(\frac{\xi_1}{\lambda},\frac{\tilde{\xi}}{\nu },\frac{\xi_n}{\nu }\right)=\displaystyle\int\limits_{\RR^n}e^{-ix\cdot\left(\frac{\xi_1}{\lambda},\frac{\tilde{\xi}}{\nu },\frac{\xi_n}{\nu }\right)}R_E^* g(x)dx\\
=\displaystyle\int\limits_{\RR^n}e^{-ix\cdot\left(\frac{\xi_1}{\lambda},\frac{\tilde{\xi}}{\nu },\frac{\xi_n}{\nu }\right)}\int\limits_{\RR^{n-1}} g\left(u, \sqrt{\frac{|u_1-x_1|^2}{\lambda^2}+\frac{|\tilde{u}-\tilde{x}|^2}{\nu^2}+\frac{x_n^2}{\nu^2}}\right)du dx\\
=\displaystyle\int\limits_{\RR^{n-1}}e^{-iu\cdot\left(\frac{\xi_1}{\lambda},\frac{\tilde{\xi}}{\nu }\right)}\int\limits_{\RR^{n}}e^{-i(x'-u,x_n)\cdot\left(\frac{\xi_1}{\lambda},\frac{\tilde{\xi}}{\nu },\frac{\xi_n}{\nu }\right)}g\left(u, \sqrt{\frac{|u_1-x_1|^2}{\lambda^2}+\frac{|\tilde{u}-\tilde{x}|^2}{\nu^2}+\frac{x_n^2}{\nu^2}}\right)dx du\\
=\displaystyle  C(\lambda)\int\limits_{\RR^{n-1}}e^{-iu\cdot\left(\frac{\xi_1}{\lambda},\frac{\tilde{\xi}}{\nu }\right)}\int\limits_{\RR^{n}}e^{-ix\cdot \xi} g(u, |x|)dx du\\
=\displaystyle (2\pi)^{n/2}C(\lambda) \int\limits_{\RR^{n-1}}e^{-iu\cdot\left(\frac{\xi_1}{\lambda},\frac{\tilde{\xi}}{\nu }\right)}(H_n  g(u,\cdot))(|\xi|) du\\
=\displaystyle (2\pi)^{n/2}C(\lambda)H_n \widehat{g}\left(\frac{\xi_1}{\lambda},\frac{\tilde{\xi}}{\nu },|\xi|\right).  
 \end{array}
\end{equation}
where $x=(x',x_n)=(x_1,\tilde{x},x_n)\in \RR^n, u=(u_1,\tilde{x})\in \RR^{n-1}$ and $\xi=(\xi',\xi_n)=(\xi_1,\tilde{\xi},\xi_n)\in \RR^n$.
Combining equation~\eqref{fourier} and equation~\eqref{backfourier}, we get equation~\eqref{inversion}.
\end{proof}

\begin{rmk}
Theorem~\ref{inversionfourier} leads naturally to a Fourier type inversion formula for even functions, if one supplements equation~\eqref{inversion} with the inverse Fourier transform.
\end{rmk}

%
%
One can also obtain a useful relation with convolution.

\begin{prop}\label{lem:unique}
Let $\phi\in C^\infty_c(\RR^{n-1}\times[0,\infty))$ and $f\in  C^\infty_c(\RR^n)$ be even in $x_n$.
If $\psi=R_E^*\phi$ and $g=\frac{1}{t^{n-1}}\frac{\partial}{\partial t}R_Ef$.
Then we have
$$
g*\phi= \dfrac{(2\pi)^{n/2}}{C(\lambda)t^{n-1}}\dfrac{\partial}{\partial t}R_E(f*\psi),
$$
where 
$$
g*\phi(u,|\omega|)=\displaystyle\int\limits_{\RR^n}\int\limits_{\RR^{n-1}}g(u-u',|\omega-\omega'|)\phi(u',|\omega'|)du'd\omega'.
$$
\end{prop}
\begin{proof}

Note that since $(2\pi)^{n/2} H_nf_0=\hat f$ for a radial function $f$ on $\RR^n$ and $f_0(|x|)=f(x)$, we get $H_n(f*g)=(2\pi)^{n/2} H_nfH_ng$.
Taking the Fourier transform of $g*\phi$ with respect to $u$ and $H_n$ with respect to $t$, we get
$$
\begin{array}{ll}
H_n\widehat{g*\phi}\left(\dfrac{\xi_1}{\lambda},\dfrac{\tilde{\xi}}{\nu },|\xi|\right)&=(2\pi)^{n/2}  H_n\hat{g}\left(\dfrac{\xi_1}{\lambda},\dfrac{\tilde{\xi}}{\nu },|\xi|\right)H_n\hat{\phi}\left(\dfrac{\xi_1}{\lambda},\dfrac{\tilde{\xi}}{\nu },|\xi|\right)\\
&=\dfrac{(2\pi)^{n}}{C(\lambda)}H_n\hat{g}\left(\dfrac{\xi_1}{\lambda},\dfrac{\tilde{\xi}}{\nu },|\xi|\right)\hat{\psi}\left(\dfrac{\xi_1}{\lambda},\dfrac{\tilde{\xi}}{\nu },\dfrac{\xi_n}{  \nu }\right).
\end{array}
$$
In the last line we used equation~\eqref{backfourier}.
Equation~\eqref{fourier} implies
$$
\begin{array}{ll}
H_n\widehat{g*\phi}\left(\dfrac{\xi_1}{\lambda},\dfrac{\tilde{\xi}}{\nu },|\xi|\right)&=\displaystyle \frac{2^{3n/2+1} \pi^{3n/2}|\xi|^{2-n}\nu^{n-2}\lambda}{C(\lambda)|\xi_n|} \hat{f}\left(\dfrac{\xi_1}{\lambda},\dfrac{\tilde{\xi}}{\nu },\dfrac{\xi_n}{  \nu }\right)\hat{\psi}\left(\dfrac{\xi_1}{\lambda},\dfrac{\tilde{\xi}}{\nu },\dfrac{\xi_n}{  \nu }\right) \\
&=\displaystyle \frac{2^{3n/2+1} \pi^{3n/2}|\xi|^{2-n}\nu^{n-2}\lambda}{C(\lambda)|\xi_n|} \widehat{f*\phi}\left(\dfrac{\xi_1}{\lambda},\dfrac{\tilde{\xi}}{\nu },\dfrac{\xi_n}{  \nu }\right)\\
&=\displaystyle \dfrac{(2\pi)^{n}}{C(\lambda)}H_n\left(\dfrac{1}{t^{n-1}} \dfrac{\partial}{\partial t}R_E\widehat{f*\phi}\right)\left(\dfrac{\xi_1}{\lambda},\dfrac{\tilde{\xi}}{\nu },|\xi|\right),
\end{array}
$$
which proves our assertion.

\end{proof}
\section{A different inversion method}\label{reddinginversion}

In this section, we provide a different inversion formula for the elliptical Radon transform.  To obtain this formula, we start to take a transform, which is like the Fourier transform, but with kernel $e^{i\omega t^2}$ instead of $e^{i\omega t}$, of the derivative of $R_Ef$ in $t$. To get the Fourier transform of $f$ from this transform, we change variables.

We start from formula~\eqref{differREF}.
Let us define 
 $$
 \begin{array}{rl}
 G(u,w)&:=\displaystyle\int\limits^\infty_0 \frac{\partial}{\partial t}R_{\mathcal{}E} f(u,t)e^{iwt^2}dt\\
&=\displaystyle C(\lambda)\int\limits^\infty_0 t^{n-1} \int\limits_{|y|=1} f(\lambda t y_1+u_1,\nu t\tilde{y}+\tilde{u},t\nu y_n)e^{iwt^2} d\sigma(y) dt\\
&=\displaystyle C(\lambda) \int\limits_{\RR^{n-1}} f(\lambda y_1+u_1,\nu \tilde{y}+\tilde{u},\nu y_n) e^{iw|y|^2}dy,
   \end{array}
$$
where in the last equality we switched from polar to Cartesian coordinates.
\begin{thm}
Let $f\in C^\infty_c(\RR^n)$ with $f(x',x_n)=f(x',-x_n)$.
Then we have
$$
f(x)=\displaystyle\dfrac{x_n }{(2\pi)^n C(\lambda)^2}\intL_{\RR^n} e^{-i\frac{|\alpha|^2}{4\gamma}} e^{i\alpha\cdot \left(\frac{x_1}{\lambda},\frac{\tilde{x}}{\nu}\right)}e^{-i\gamma \left(\frac{x_1^2}{\lambda^2}+\frac{|\tilde{x}|^2}{\nu^2}+\frac{x_n^2}{\nu^2}\right)}G\left(\dfrac{\alpha_1\lambda}{2\gamma},\dfrac{\alpha'\nu}{2\gamma},\gamma\right) d\alpha d\gamma,
$$
 for $x_n>0$, where $C(\lambda)=\lambda\nu^{n-1}$, as before.

 \end{thm}
 \begin{proof}
Making the change of variable $x_1=\lambda t y_1+u_1,\tilde{x}=\nu t\tilde{y}+\tilde{u},x_n=t\nu y_n$, we get
 \begin{equation*}
  \begin{array}{ll}
 G(u,w)&=\displaystyle \int\limits_{\RR^n} f(x)e^{iw\left(\frac{(x_1-u)^2}{\lambda^2}+\frac{|\tilde{x}-\tilde{u}|^2}{\nu^2}+\frac{x_n^2}{\nu^2}\right)}dx\\
 &=e^{iw\frac{u^2}{\lambda^2}}e^{iw\frac{|\tilde{u}|^2}{\nu^2}}\displaystyle \int\limits_{\RR^n} f(x)e^{iw\left(\frac{x_1^2}{\lambda^2}+\frac{|\tilde{x}|^2}{\nu^2}+\frac{x_n^2}{\nu^2}\right)} e^{-2iwu_1\frac{x_1}{\lambda^2}}e^{-2iw\frac{\tilde{u}\cdot \tilde{x}}{\nu^2}}dx,
  \end{array}
 \end{equation*}
 where $x=(x_1,\tilde{x},x_n)$ and $u=(u_1,\tilde{u})\in \RR^{n-1}$.
 Next, make the change of variables 
 $$
 \mathsf{x}_{1}=\dfrac{x_1}{\lambda},\qquad \tilde{\mathsf x}=\dfrac{\tilde{x}}{\nu },\qquad \mbox{ and }\qquad r=\dfrac{x^2}{\lambda^2}+\dfrac{|\tilde{x}|^2}{\nu^2}+\dfrac{x_n^2}{\nu^2},
 $$
  so that 
 \begin{equation*}
  \begin{array}{ll}
 x_1=\mathsf x_1\lambda,\qquad \tilde{x}=\tilde{\mathsf x}\nu ,\qquad\mbox{ and }\qquad x_n=\nu \sqrt{r-\mathsf x^2_1-|\tilde{\mathsf x}|^2}.
  \end{array}
 \end{equation*}
 The Jacobian of this transformation is
 $$
 J=\left|\begin{array}{cccc}\lambda  & 0 &\cdots & 0\\ 0 &\nu  &\cdots &0 \\ \vdots& \vdots &\ddots& \vdots \\  \frac{-\mathsf x_{1}\nu }{2\sqrt{r-\mathsf x^2_1-\tilde{\mathsf x}^2}} &\frac{-\tilde{\mathsf x}\nu }{2\sqrt{r-\mathsf x^2_1-|\tilde{\mathsf x}|^2}} &\cdots &\frac{\nu }{2\sqrt{r-\mathsf x^2_1-|\tilde{\mathsf x}|^2}}\end{array}\right|=\dfrac{C(\lambda)}{2\sqrt{r-\mathsf x^2_1-|\tilde{\mathsf x}|^2}}
 $$
 so that
 $$
 dx=\dfrac{C(\lambda)}{2\sqrt{r-\mathsf x^2_1-|\tilde{\mathsf x}|^2}}dx_1d\tilde{x} dr.
 $$
 Let the function $k(x,\tilde{x},r)$ be defined by
 \begin{equation*}
  k(x,\tilde{x},r)=\left\{\begin{array}{ll}\dfrac{f(\lambda x_1,\nu \tilde{x}, \nu \sqrt{r-x_1^2-|\tilde{x}|^2})}{2\sqrt{r-x_1^2-|\tilde{x}|^2}}&0<|x_1|^2+|\tilde{x}|^2<r,\\
                 0 & \mbox{otherwise}.\end{array}\right. 
 \end{equation*}
 Since $f$ is even in $x_n$, it is sufficient to consider the positive root of $\sqrt{r-x_1^2-|\tilde{x}|^2}.$
 Then we can rewrite $G(u,w)$ as
  \begin{equation*}
  \begin{array}{rl}
 G(u,w)&=C(\lambda) e^{iw\frac{u_1^2}{\lambda^2}}e^{iw\frac{|\tilde{u}|^2}{\nu^2}}\displaystyle \intL_{\RR^n} k(x_1,\tilde{x},r)e^{iwr}e^{-2i\frac{wx_1u}{ \lambda}}e^{-2i\frac{w\tilde{x}\cdot \tilde{u}}{\nu}}dx_1d\tilde{x}dr\\\
 &=C(\lambda) e^{iw\frac{u_1^2}{\lambda^2}}e^{iw\frac{|\tilde{u}|^2}{\nu^2}} K\left(2\dfrac{wu_1}{\lambda},2\dfrac{w\tilde{u}}{\nu},-w\right),
  \end{array}
 \end{equation*}
 where for $\alpha=(\alpha_1,\alpha')\in \RR\times\RR^{n-2}$,
 \begin{equation*}
  \begin{array}{ll}
 K(\alpha,\gamma)&=\displaystyle\intL_{\RR^n} e^{-i\alpha\cdot(x_1,\tilde{x})}e^{-i\gamma r}k(x_1,\tilde{x},r)dx_1d\tilde{x}dr\\\
 &=\dfrac{1}{C(\lambda)}e^{i\frac{|\alpha|^2}{4\gamma}}G\left(\dfrac{-\alpha_1\lambda}{2\gamma},\dfrac{-\alpha'\nu}{2\gamma},-\gamma\right).  
  \end{array}
 \end{equation*}

 Since $k(x,\tilde{x},r)$ is
 $$\dfrac{1}{(2\pi)^n}\displaystyle\intL_{\RR^n} e^{i\alpha_1 x_1}e^{i\alpha'\cdot \tilde{x}}e^{i\gamma r}K(\alpha,\gamma)d\alpha d\gamma,$$
 we get for $x_n>0$,
 \begin{equation}\label{redding}
  \begin{array}{ll}
   f(x)&=\dfrac{x_n }{C(\lambda)} k\left(\dfrac{x_1}{\lambda},\dfrac{\tilde{x}}{\nu},\dfrac{x_1^2}{\lambda^2}+\dfrac{|\tilde{x}|^2}{\nu^2}+\dfrac{x_n^2}{\nu^2}\right)\\\
 &=\displaystyle\dfrac{x_n }{(2\pi)^nC(\lambda)}\intL_{\RR^n} e^{i\alpha_1 \frac{x_1}{\lambda}}e^{i\alpha'\cdot \frac{ \tilde{x}}{\nu}}e^{i\gamma \left(\frac{x_1^2}{\lambda^2}+\frac{|\tilde{x}|^2}{\nu^2}+\frac{x_n^2}{\nu^2}\right)}K(\alpha,\gamma)d\alpha d\gamma\\\
&=\displaystyle\dfrac{x_n }{(2\pi)^n C(\lambda)^2}\intL_{\RR^n} e^{-i\frac{|\alpha|^2}{4\gamma}} e^{i\alpha\cdot \left(\frac{x_1}{\lambda},\frac{\tilde{x}}{\nu}\right)}e^{-i\gamma \left(\frac{x_1^2}{\lambda^2}+\frac{|\tilde{x}|^2}{\nu^2}+\frac{x_n^2}{\nu^2}\right)}G\left(\dfrac{\alpha_1\lambda}{2\gamma},\dfrac{\alpha'\nu}{2\gamma},\gamma\right) d\alpha d\gamma.
  \end{array}
 \end{equation}
\end{proof}
\section{Stability estimate}\label{estimate}
In this section, we obtain a stability estimate for the elliptical Radon transform. Let $\mathcal{H}^{\gamma}(\RR^n)$ be a regular Sobolev space with a norm
$$
||f||^2_{\gamma}:=\int\limits_{\RR^n}|\hat{f}(\xi)|^2(1+|\xi|^2)^\gamma d\xi.
$$
Let us define $\mathcal{H}^\gamma_e(\RR^n)=\{f\in \mathcal{H}^\gamma(\RR^n): f \mbox{ is even in }x_n\}$ and let $L^2_{n-1}(\RR^{n-1}\times[0,\infty))$ be the set of a fucntion $g$ on $\RR^{n-1}\times[0,\infty)$ with
$$
||g||^2:=\intL_{\RR^{n-1}}\intL\half |g(u,t)|^2t^{n-1}dtdu<\infty.
$$
Then $L^2_{n-1}(\RR^{n-1}\times[0,\infty))$ is a Hilbert space. Also, by the Plancherel formula, we have $||g||=(2\pi)^{2n-1}||\tilde g||$, where
$$
\tilde g(\xi,|\zeta|)=\intL_{\RR^{n-1}}\intL_{\RR^{n}}g(u,|w|)e^{-i(u,w)\cdot(\xi,\zeta)}dpdw.
$$
Let $\mathcal H^\gamma(\RR^{n-1}\times[0,\infty))$ be the set of a function $g\in L^2_{n-1}(\RR^{n-1}\times[0,\infty))$ with $||g||_{\gamma}<\infty$, where
$$
||g||_\gamma^2:=\intL_{\RR^{n-1}}\intL\half|\tilde g(\xi',\eta)|^2(1+|\xi'|^2+|\eta|^2)^\gamma \eta^{n-1}d\eta d\xi'.
$$
\begin{thm}\label{norm}
For $\gamma\geq0$, there is a constant $C_n$ such that $f\in \mathcal H^\gamma_e(\RR^n)$,
$$
||f||_\gamma\leq C_n||t^{1-n}\partial_tR_Ef||_{\gamma+(n-1)/2}.
$$
\end{thm}
\begin{proof}
Let $g=t^{1-n}\partial_tR_Ef$.
Note that from equation~\eqref{backfourier}, we have
\begin{equation}\label{eq:hatandtilde}
\widehat{R^*_Eg}\left(\frac{\xi_1}\lambda,\frac{\tilde \xi}\nu,\frac{\xi_n}\nu \right)=C(\lambda)\intL_{\RR^{n-1}}e^{-iu\cdot\left(\frac{\xi_1}\lambda,\frac{\tilde \xi}\nu\right)}\intL_{\RR^n}e^{-ix\cdot \xi}g(u,|x|)dxdu=C(\lambda)\tilde g\left(\frac{\xi_1}\lambda,\frac{\tilde \xi}\nu,|\xi|\right).
\end{equation}
Combining this equation and Theorem~\ref{inversion}, we have 
$$
\hat{f}(\xi)=\dfrac{|(\lambda\xi_1,\nu \tilde{\xi},\nu \xi_n)|^{n-2}|\nu\xi_n|}{2^{n+1}\pi^nC(\lambda)}\tilde g\left(\xi_1,\tilde \xi,|(\lambda\xi_1,\nu\tilde \xi,\nu\xi_n)|\right).
$$
Hence, we have
$$
\begin{array}{ll}
||f||_\gamma^2&\displaystyle=\intL_{\RR^n}(1+|\xi|^2)^\gamma|\hat f(\xi)|^2 d\xi\\
&\displaystyle=\dfrac{1}{2^{2n+2}\pi^{2n}C(\lambda)^2}\intL_{\RR^n}|(\lambda\xi_1,\nu \tilde{\xi},\nu \xi_n)|^{2n-4}|\nu\xi_n|^2(1+|\xi|^2)^\gamma |\tilde g(\xi_1,\tilde \xi,|(\lambda\xi_1,\nu\tilde \xi,\nu\xi_n)|)|^2d\xi\\
&\displaystyle\leq C_n\intL_{\RR^n}|(\lambda\xi_1,\nu \tilde{\xi},\nu \xi_n)|^{2n-4}|\nu\xi_n|^2(1+|(\lambda\xi_1,\nu \tilde{\xi},\nu \xi_n)|^2)^\gamma |\tilde g(\xi_1,\tilde \xi,|(\lambda\xi_1,\nu\tilde \xi,\nu\xi_n)|)|^2d\xi\\
&\displaystyle\leq C_n\intL_{\RR^{n-2}}\intR\intL^\infty_0(\eta^2-\lambda^2\xi_1^2-\nu^2|\tilde\xi|^2)^\frac12\eta^{2n-3}(1+\eta^2)^\gamma |\tilde g(\xi_1,\tilde \xi,\eta)|^2d\eta d\xi_1d\tilde\xi.
\end{array}
$$
In the last line, we change the variable $\xi_n$ to $\eta=|(\lambda\xi_1,\nu \tilde{\xi},\nu \xi_n)|$.
\end{proof}
\section{Uniqueness for the local problem}\label{unique}

Theorem~\ref{inversion} implies that an even function $f\in C^\infty_c(\RR^n)$ is uniquely determined by $R_Ef$. The question arises if $f$ is uniquely determined by some partial information. The approach in this section is similar to the one in~\cite{andersson88}.
\begin{thm}
Let $u^0\in\RR^{n-1}$, $\epsilon>0$, and $T>0$ be arbitrary.
Let $f\in C^\infty_c(\RR^n)$ be even in $x_n$ and suppose $g=R_Ef$ is equal to zero on the open set 
$$
U_{T,\epsilon}=\{(u,t)\in\RR^{n-1}\times[0,\infty):|u-u^0|<\epsilon,0\leq t<T\}.
$$
Then $f$ equals zero on the open set 
$$
V_T=\left\{x\in\RR^n:\frac{(x_1-u^0_{1})^2}{\lambda^2}+\frac{(\tilde{x}-\tilde{u}^0)^2}{\nu^2}+\frac{x_n^2}{\nu^2}<T^2\right\}.
$$
Here $x=(x_1,\tilde{x},x_n)\in\RR^n$ and $u^0=(u^0_{1},\tilde{u}^0)\in\RR^{n-1}$.
Also, $g$ is equal to zero on the open cone 
$$
W_T=\{(u,t)\in\RR^{n-1}\times[0,T):|u-u_0|+ t<T\}.
$$
\end{thm}
\begin{figure}
\centering
 \includegraphics[width=0.37\textwidth]{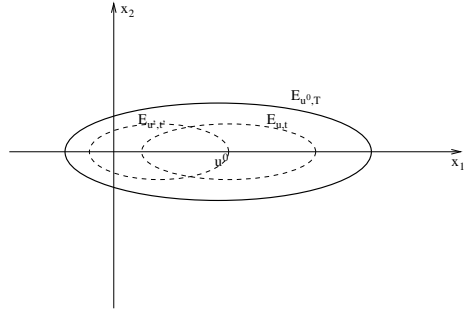} 
 \caption{Ellipses $E_{u^0,T},E_{u,t},E_{u',t'}$}      
\end{figure}

\begin{proof}
Without loss of generality, we may assume $u^0=0$.
Let $f\in C^\infty(\RR^n)$.
Clearly, $g$ is also differentiable.
Differentiating $R_Ef(u,t)$ with respect to $u_i$ yields
$$
\begin{array}{ll}
\dfrac{\partial}{\partial u_i}R_E f(u,t)&=C(\lambda)\displaystyle \int\limits_{\RR^n}\chi_{|x|<t}\dfrac{\partial}{\partial u_i} f(\lambda x_1+u_1,\nu \tilde{x}+\tilde{u},\nu x_n) dx\\
&=C(\lambda)\displaystyle \frac{1}{t}\int\limits_{|x|=t}x_i\dfrac{\partial}{\partial x_i} f(\lambda x_1+u_1,\nu \tilde{x}+\tilde{u},\nu x_n) dx.
\end{array}
$$
Here we used equation~\eqref{ellipticalradon} and the divergence theorem.
Using equation~\eqref{differREF}, we get
$$
\begin{array}{ll}
\dfrac{\partial}{\partial t}R_E (x_i f)(u,t)&=C(\lambda)\displaystyle \int\limits_{|x|=t}(u_i+\nu x_i) f(\lambda x_1+u_1,\nu \tilde{x}+\tilde{u},\nu x_n) d\sigma(x)\\
&=C(\lambda)\displaystyle \left(t\nu \dfrac{\partial}{\partial u_i}g(u,t) +u_i\dfrac{\partial}{\partial t}g(u,t)\right).
\end{array}
$$
Let the linear operator $D_i$ be defined by $D_ig(u,t)=C(\lambda)(t\nu\partial_{u_i}g(u,t)+u_i\partial_tg(u,t))$.
Then $\frac{\partial}{\partial t}R_E (x_i f)(u,t)$ is $D_i g(u,t)$.
By iteration, we obtain $\frac{\partial}{\partial t}R_E(p(x')f)=p(D)g$ where $p$ is an $n-1$-variable polynomial.
If $g$ is zero in $U_{T,\epsilon}$, then $p(D)g$ is also zero in $U_{T,\epsilon}$.
Then we have for any point $(u,t)\in U_{T,\epsilon}$,
$$
\begin{array}{ll}
\dfrac{\partial}{\partial t}R_E (p(x') f)(u,t)&=C(\lambda)\displaystyle \int\limits_{|x|=t}p(u_1+\lambda x_1,\tilde{u}+\nu \tilde{x}) f(\lambda x_1+u_1,\nu \tilde{x}+\tilde{u},\nu x_n) d\sigma(x)\\
&=C(\lambda)\displaystyle \int\limits_{|y|<t}p(u+(\lambda y_1,\nu \tilde{y})) f(u+(\lambda y_1,\nu \tilde{y}),\nu \sqrt{t^2-|y|^2}) \dfrac{dy}{\sqrt{t^2-|y|}}=0.
\end{array}
$$
For fixed $u$ and $t$, choose a sequence of polynomials such that $p_i(u+(\lambda y_1,\nu \tilde{y}))$ converge to $ f(u+(\lambda y_1,\nu \tilde{y},\nu \sqrt{t^2-|y|^2}))$ uniformly for $|y|\leq t$ and $y=(y_1,\tilde{y})\in\RR^{n-1}$.
It follows that $f=0$ in $V_T$ and that $g=0$ in $W_T$.
\end{proof}

\section{Conclusion}

We describe two different ways of determining a function $f$ from its $n$-dimensional elliptical Radon transform $R_Ef$ arising in radio tomography imaging~\cite{wilsonp10,wilsonp09,wilsonpv09}. 
We also present a stability estimate and a local uniqueness results for $R_E$. 

\section*{Acknowledgements}
The author thanks P Kuchment and D Steinhauer for fruitful discussions.
This work was supported in part by US NSF Grants DMS 0908208 and DMS 1211463.
\end{document}